\documentclass[reqno]{amsart}
\usepackage{fixltx2e}
\usepackage{amssymb,amsmath,amsfonts,amsthm}
\usepackage[utf8]{inputenc}
\usepackage[T1]{fontenc}
\usepackage[all]{xy}
\usepackage{mathbbol}
\usepackage{mathabx}
\usepackage{graphicx}
\usepackage{xcolor}

\usepackage{tikz-cd} 

\newtheorem{thm}{Theorem}[section]
\newtheorem{cor}[thm]{Corollary}

\newtheorem{prop}[thm]{Proposition}
\theoremstyle{definition}
\newtheorem{defn}[thm]{Definition}

\newtheorem{ex}[thm]{Example}

\begin{document}
	
\title[Bipartite graphs and finite-dimensional semisimple Leibniz algebras]{Bipartite graphs and the structure of finite-dimensional semisimple Leibniz algebras}
\author{R. Turdibaev}
\address{[Rustam Turdibaev] Inha University in Tashkent, Ziyolilar 9, 100170 Tashkent, Uzbekistan}
\email{r.turdibaev@inha.uz}

\begin{abstract}  Given a finite connected bipartite graph, finite-dimensional indecomposable semisimple Leibniz algebras are constructed. Furthermore, any finite-dimensional indecomposable semisimple Leibniz algebra admits a similar construction.

\end{abstract}

\maketitle

	\section{Introduction}
\indent 

The finite-dimensional simple Lie algebras over an algebraically closed field of characteristic zero are classified and it is a classical result that a finite-dimensional semisimple Lie algebra is a direct sum of simple Lie subalgebras (see \cite{Jacobson}). A finite-dimensional module over a semisimple Lie algebra due to Weyl's theorem is completely reducible into a direct sum of simple submodules. Furthermore, a Lie algebra admits a Levi decomposition -- a semi-direct sum of a semisimple subalgebra and a maximal solvable ideal. 
	
In this paper a ``non-commutative'' generalization of Lie algebras, introduced by Bloh (\cite{Bloh}) and later by Loday (\cite{LodayCyclic},\cite{LodayPaper})- the so-called Leibniz algebras are studied. Although the classical simplicity for Leibniz algebras implies that it is a Lie algebra, a modified definition of the simplicity was introduced in \cite{Dzumadil'daev} and has been in use in the various papers on the structure theory of Leibniz algebras. Generalization of semisimplicity for Leibniz algebras draws a parallel with semisimple Lie algebras, which is the main focus of the current work. However, it is well-known that a semisimple Leibniz algebra is not in general a direct sum of simple Leibniz algebras, and the question on the structure of semisimple Leibniz algebras has been open. Recently, in \cite[Theorem 3.5]{AKOZh} the authors establish the description of  the finite-dimensional indecomposable semisimple non-Lie Leibniz algebra using a graph, whose vertices are the simple Lie subalgebras of the liezation of the Leibniz algebra. 

Motivated by the results of \cite{AKOZh}, the goal of this work is to shed light on how finite-dimensional semisimple Leibniz algebras are built. Turns out, the structure of a  finite-dimensional semisimple Leibniz algebra is more clear if instead of a graph (cf \cite{AKOZh}), one uses a bipartite graph. The main results are presented in the last section, consisting of the description of semisimple Leibniz algebras from \cite{AKOZh} with different proofs and a new construction of finite-dimensional semisimple Leibniz algebras using the bipartition of an associated graph.

\section{Preliminaries}
In the following section necessary definitions and results on Lie algebra and representation theory is given. Connection of Leibniz algebra with Lie algebra and its modules is provided. One of the main ingredients in this work, an analogue of Levi's theorem for Leibniz algebra is obtained directly from the results of T. Pirashvili \cite{Pirashvili} in subsection \ref{section_levi}. 
\subsection{Lie algebras and Leibniz algebras}
\begin{defn}
	An algebra $\mathfrak{g}$ over $\mathbb{K}$ is called a Lie  algebra over $\mathbb{K}$ if its multiplication ($[-,-]$) satisfies the identities:
	
	(1) $[x,x]=0$
	
	(2) $[x,[y,z]]+[y,[z,x]]+[z,[x,y]]=0$ for all $x,y,z\in \mathfrak{g}$ (\textit{Jacobi} identity).
\end{defn}
	
\begin{defn}
	A (left) $\mathfrak{g}$-module is vector space $  M $ together with a map $\mathfrak{g}\otimes   M  \to   M,\ x\otimes   m  \mapsto x.  m $   such that for all $x,y \in \mathfrak{g}$ and all $m\in M $ we have  
	\begin{center}
		$[x,y].m=x.(y.m)-y.(x.m).$
	\end{center}
\end{defn}	
	Given a  left   $\mathfrak{g}$-module action on $M$, one can construct a  right action by $m.x=: -x.m  $ and it satisfies the identity 
	\begin{center}
		$m.[x,y] =(m.x).y-(m.y).x.$
	\end{center}
 
	Left and right actions induce Lie algebra structure on $M\oplus \mathfrak{g}$, where $M$ becomes an abelian ideal and $\mathfrak{g}$ is a subalgebra. 
If one has the right action of $\mathfrak{g}$ on $M$ and sets the left action to be zero, this induces a new type of a product that generalizes the Lie bracket on $\mathfrak{g}$ given in the following 
\begin{ex}\label{onlyexample} (\cite{Kinyon})
	Let $\mathfrak{g}$ be a Lie algebra and $M$ be a $\mathfrak{g}$-module. Consider $L=\mathfrak{g}\oplus M$ with a bracket $[(g_1,m_1), (g_2,m_2)]:=([g_1,g_2], -g_2.m_1).$ Then  \begin{equation}\label{Leibnizidentity} [[x,y],z]=[[x,z],y]+[x,[y,z]]\end{equation}  
	holds for any $x,y,z\in L$ and this algebra is not a Lie algebra if the action of $\mathfrak{g}$ on $M$ is not trivial.
\end{ex}
A similar construction is given in \cite{Dzumadil'daev}. 
\begin{defn}
	A vector space $L$ with a bracket that satisfies identity (\ref{Leibnizidentity}) is called a \textit{Leibniz} algebra.
\end{defn}

The Leibniz algebra of Example \ref{onlyexample} is denoted by $\mathfrak{g}\ltimes M$.

For a Leibniz algebra $L$ set $I=\textrm{Span} \langle [x,x] \mid x\in L\rangle .$ Then $I$ is a proper ideal of  $L$ and $[L,I]=\{0\}$. If $I=\{0\}$ then $L$ is a Lie algebra (and the converse is also true). For a non-Lie Leibniz algebra the ideal $I$ is always non-trivial and the following notion is well-defined. 

\begin{defn}\label{liezation}
	For a Leibniz algebra $L$ the quotient algebra $\mathfrak{g}_L:=L/I$ is a Lie algebra which is called the \textit{liezation} of $L$.
\end{defn}
 There is the  following short exact sequence and the epimorphism $f$ is universal in the sense that a Leibniz algebra homorphism from $L$ to any Lie algebra $\mathfrak{g}$ factors through $f$: 
 \begin{center}
	\begin{tikzcd}
		I \arrow[r,hook] &L \arrow{d}\arrow[r, two heads]{r}{f}
		&\mathfrak{g}_L \arrow[dashed]{ld}\\
		& \mathfrak{g}
	\end{tikzcd}  
\end{center}
Note that, due to $I$ annihilating the Leibniz algebra whenever multiplied form the right, $I$  admits a structure of a  right Lie algebra $\mathfrak{g}_L$-module with the well-defined action $i.g=[i,s(g)]$, where $s:\mathfrak{g}_L \to L$ is a linear section. 	  

\begin{defn}
	An algebra $A$ is called \textit{decomposable} if $A=A_1\oplus A_2$ for some proper ideals. An algebra without this property is called   \textit{indecomposable}.
\end{defn}

\noindent Any finite-dimensional algebra is either indecomposable or is a finite direct sum of indecomposable algebras. A Lie algebra with no non-trivial ideals is called  \textit{simple}. Simple Lie algebras are indecomposable, while the converse is not necessarily true. 

A Leibniz algebra with only one non-trivial ideal $I$, so-called   \textit{simple}   Leibniz algebra, is indecomposable.

\subsection{Levi's Theorem for Leibniz algebras}\label{section_levi}

Similarly, as in Lie algebra theory, the following notions transfer to the case of Leibniz algebra. 

\begin{defn}
 An ideal $S$ of a Leibniz algebra $L$ is called  \textit{solvable}  if $S^{[k]}=\{0\}$ for some integer $k$, where $S^{[0]}=S, \ S^{[m+1]}=[S^{[m]},S^{[m]}]$. 
The maximum solvable ideal is called solvable \textit{radical} and is denoted by $\textrm{Rad}(\mathfrak{g})$.
\end{defn}

 A Lie algebra $\mathfrak{g}$ is called  \textit{semisimple}   if $\textrm{Rad}(\mathfrak{g})=\{0\}$. There is a well-known Levi-Malcev decomposition of a finite-dimensional Lie algebra $\mathfrak{g}$ as a semidirect sum of a semisimple subalgebra $\mathfrak{s}$ and the solvable radical   {$\mathfrak{g}=  \mathfrak{s}  \ltimes \textrm{Rad}(\mathfrak{g})$}. For Leibniz algebras similar result was proved by D.~Barnes \cite{Barnes} in 2011. Note that, the same result is implicit from \cite[Proposition 2.4]{Pirashvili} given below. 
\begin{prop}\label{section}
	 Let  $\phi \colon L \to \mathfrak{g}$ be an epimorphism from an arbitrary finite-dimensional Leibniz algebra $L$ to a   semisimple   Lie algebra $\mathfrak{g}$. Then $\phi$ admits a   section.
\end{prop}
Indeed, consider a finite-dimensional Leibniz algebra $L$ and apply Levi-Malcev decomposition to its liezation $\mathfrak{g}_L$. Applying Proposition \ref{section} for an epimorphism $g\circ f$ one obtains a section: 
\begin{center}
	    \begin{tikzcd}[ampersand replacement=\&]
			L\arrow[two heads]{r}{f}\arrow[leftarrow, dashed, ]{dr}{}
			\&\mathfrak{g}_L=\mathfrak{s}\ltimes \text{rad}(\mathfrak{g})\arrow[two heads]{d}{g}\\
			\&\mathfrak{s}
		\end{tikzcd}    
	\end{center}
Clearly, the kernel of the epimorphism $g\circ f$ is $\textrm{Rad}(L)$ and we have an analogue of Levi decomposition for Leibniz algebra $L\cong \mathfrak{s}\ltimes \textrm{Rad}(L)$. The Malcev part of the theorem is not true in general for the case of Leibniz algebras as shown in \cite{Barnes}. In some cases, conjugacy of Levi subalgebras is possible (see \cite{onLevi1} and \cite{onLevi2}).
	 
\begin{defn}
A Leibniz algebra $L$ is called \textit{semisimple} if its liezation $\mathfrak{g}_L$ is a semisimple Lie algebra.
\end{defn}
\noindent Note that, from Levi's decomposition it follows that $L$ is a semisimple Leibniz algebra if and only if $\text{Rad}(L)=I$. Furthermore, there is the following 
\begin{cor}\label{semisimple_corollary}
	  Let $L$ be a finite-dimensional   semisimple   Leibniz algebra. Then  $L\cong \mathfrak{g}_L \ltimes I$, where $\mathfrak{g}_L$ is a semisimple Lie algebra (liezation of $L$). 
\end{cor}

Since $I$ is a $\mathfrak{g}_L$-module, and over a semisimple Lie algebra by Weyl's semisimplicity $I$ decomposes into a direct sum of simple $\mathfrak{g}_L$-submodules, one obtains
\begin{equation}\label{decomposition}
	\displaystyle L= (\oplus_{i=1}^m\mathfrak{g}_i)\ltimes (\oplus_{k=1}^n I_k)
\end{equation}
where $\mathfrak{g}_i$'s are simple Lie algebras and $I_k$'s are simple $\oplus_{i=1}^m\mathfrak{g}_i$-modules. 

Let a semisimple Leibniz algebra $L$ be decomposable, that is $L=L_1\oplus \ldots \oplus L_t$ and $L_1,\dots, L_t$ are indecomposable Leibniz algebras. Obviously, $L_1$ is also semisimple and Corollary \ref{semisimple_corollary} implies  $L_1=\mathfrak{g}_1\ltimes I_1$, where $\mathfrak{g}_1$ as the liezation of $L_1$ must be a subalgebra of $\mathfrak{g}_L$. Thus, it is a direct sum of some simple components of $\mathfrak{g}_L$. Since $L_1\unlhd L$, then $I_1$ is a $\mathfrak{g}_L$-module and using $I_1\subseteq I$ it is a sum of simple $\mathfrak{g}$-submodules of $I$. This implies that not only $L_1$ admits the structure of decomposition (\ref{decomposition}), but $L_1=(\oplus_{i\in A} \mathfrak{g}_i)\ltimes (\oplus_{k\in B} I_k)$ for some $A\subseteq\{1,\dots, m\},\ B\subseteq\{ 1,\dots, n\}$. Hence, the study of the structure of a semisimple Leibniz algebra is reduced to the study of an indecomposable semisimple Leibniz algebra. \\

We use the following result from \cite[Theorem 6]{bookWan} on the structure of a simple Lie module over a semisimple Lie algebra. 

\begin{thm}\label{theorem_Wan}
	Let $M$ be a finite-dimensional simple module over a finite-dimensional semisimple Lie algebra $\mathfrak{g}=\displaystyle \oplus_{i=1}^n \mathfrak{g}_i$. Then $M\cong \otimes_{i=1}^n M_i$, where $M_i$ is a simple  $\mathfrak{g}_i$-module for all $i=1,\dots,n$.
\end{thm}
\section{Main Results}

\begin{prop}\label{Ig=I or 0} Let $L$ be a finite-dimensional semisimple Leibniz algebra. Then $[I_k,\mathfrak{g}_i]=I_k$ or $\{0\}$ for all indexes $i$ and $k$ of the decomposition (\ref{decomposition}).
\end{prop}
\begin{proof}
	Without loss of generality let us consider $[I_1,\mathfrak{g}_1]$.  Since $I_1$ is a simple $\oplus_{i=1}^n \mathfrak{g}_i$-module, by Theorem \ref{theorem_Wan} $I_1=\otimes_{i=1}^n J_i$ for simple $\mathfrak{g}_i$-modules $J_i$. Note that the action is $[I_1,\mathfrak{g}_1]=[J_1\otimes \dots \otimes J_n,\mathfrak{g}_1]=(J_1\otimes \dots \otimes J_n).\mathfrak{g}_1=(J_1.\mathfrak{g}_1)\otimes J_2\otimes\dots \otimes J_n$. Now this is either $\{ 0\}$ or $J_1\otimes\dots \otimes J_n=I_1$ since $J_1$ is an irreducible $\mathfrak{g}_1$-module.    Moreover, it is well-known from representation theory of Lie algebras that $[\mathfrak{g}_1, I_1]=\{0\}$ if and only if $J_1=\mathbb{C}$ (a one-dimensional representation is trivial). 
\end{proof}

\begin{defn}(\cite{AKOZh})
	Let $L$ be a semisimple Leibniz algebra with decomposition (\ref{decomposition}). Two Lie algebras $\mathfrak{g}_i$ and $\mathfrak{g}_j$ from decomposition (\ref{decomposition}) are called \textit{adjacent} if there exists $k$ such that $[I_k,\mathfrak{g}_i]=I_k=[I_k,\mathfrak{g}_j]$. 
\end{defn}

It is implicit from this definition that for a semisimple Leibniz algebra $L$ with decomposition (\ref{decomposition}), a graph $\Gamma$  with vertexes $\{\mathfrak{g}_1,\dots, \mathfrak{g}_m\}$ is built.  Two vertexes $\mathfrak{g}_i$ and $\mathfrak{g}_j$ are connected by an edge if  $[I_k,\mathfrak{g}_i]=I_k=[I_k,\mathfrak{g}_j]$ for some $k$. The description of indecomposable semisimple Leibniz algebra is established in \cite[Theorem 3.5]{AKOZh} and rephrased in terms of connectivity of the graph $\Gamma$ below.
	
\begin{thm}\label{graph_ideal}
	Let $L$ be an indecomposable semisimple Leibniz algebra with decomposition (\ref{decomposition}). Then $[I_k, \oplus_{i=1}^m \mathfrak{g}_i]=I_k$ for all $1\leq k\leq n$ and $\Gamma$ is a connected graph.
\end{thm}

\begin{proof} If for some $1\leq k \leq n$ one has $[I_k,\mathfrak{g}_i]=\{0\}$ for all $1\leq i \leq m$, then $I_k$ is a direct summand of $L$ (in fact, it is a contradiction with $I_k\subseteq I$ being generated by the squares). Thus, from Proposition \ref{Ig=I or 0} it follows that $[I_k,\mathfrak{g}_i]=I_k$ for some $1\leq i \leq m$ which implies the first part of the statement. If graph $\Gamma$ is disconnected, let $\{\mathfrak{g}_i\}_{i\in A}$ be some connected component of $\Gamma$. Using Proposition \ref{Ig=I or 0} build a set of indexes $B$, where for $p\in B$ there is some $q\in A$ such that $[I_p,\mathfrak{g}_q]=I_p$. Then $(\oplus_{i\in A} \mathfrak{g}_i)\ltimes (\oplus_{k\in B} I_k)$ is a direct summand of the Leibniz algebra $L$, which is a contradiction.
\end{proof}
Using decomposition (\ref{decomposition}) consider another associated graph on the vertexes $\{\mathfrak{g}_1,\dots, \mathfrak{g}_m,I_1,\dots, I_n \}$, in which two vertices are connected via an edge whenever the bracket of the end-points is a non-zero set. Note that, due to $[I_p,I_q]= [\mathfrak{g}_i,\mathfrak{g}_j]=\{0\}$, the only edges are from the set of simple Lie algebras $\{\mathfrak{g}_1,\dots, \mathfrak{g}_m\}$ to the set of simple $\mathfrak{g}_L$-modules $\{I_1,\dots, I_n\}$. 

\begin{defn} For a semisimple Leibniz algebra $L$ define its  \textit{corresponding} undirected bipartite graph B$\Gamma$ using decomposition (\ref{decomposition}) with bipartition $(\mathcal{I},\mathcal{G})$, where $\mathcal{I}=\{I_1,\dots, I_n\}, \ \mathcal{G}=\{\mathfrak{g}_1,\dots, \mathfrak{g}_m\}$ and $I_k\mathfrak{g}_i$ is an edge if and only if $[I_k,\mathfrak{g}_i]=I_k$.
\end{defn}

\begin{thm}\label{indecomposable_iff_connected}
A finite-dimensional semisimple Leibniz algebra is indecomposable if and only if the associated bipartite graph   $\textrm{B}\Gamma$ is connected.
\end{thm}

\begin{proof}
Assume that B$\Gamma$ is not connected. Let $(\{g_i\}_{i\in A}, \{I_k\}_{k\in B})$ be one of the connected components of B$\Gamma$.   Then $(\bigoplus_{i\in A} g_i )\ltimes (\bigotimes_{k\in B} I_k)$ is a direct summand of the Leibniz algebra, thus the algebra is decomposable.  

Conversely, if the Leibniz algebra is decomposable, then the corresponding bipartite graph is disconnected.  
\end{proof} 

The following corollary establishes that the converse of Theorem \ref{graph_ideal} is also valid.

\begin{cor} For a semisimple Leibniz algebra the following conditions are equivalent:
	
	$(i)$ $[I_k,\oplus_{i=1}^m \mathfrak{g}_i]=I_k$ for all $k=1,\dots n$ and graph   $\Gamma$   is connected;
	
	$(ii)$ B$\Gamma$ is connected. 
\end{cor}

\begin{proof} By Theorems \ref{indecomposable_iff_connected} and \ref{graph_ideal} it follows that $(ii)$ implies $(i)$. Now if $(i)$ holds, then the proof of Theorem \ref{graph_ideal} implies that graph B$\Gamma$ is connected.
\end{proof}

Presented example below show that the first condition of the statement $(i)$ is essential.

\begin{ex} Consider a decomposition $(\oplus_{i=1}^4 \mathfrak{g}_i)\ltimes (I_1\oplus I_2 \oplus I_3)$ with the non-zero products:
$$[I_3,\mathfrak{g}_1]=[I_3,\mathfrak{g}_2]=I_3 \text{ and } [I_2,\mathfrak{g}_2]=[I_2,\mathfrak{g}_3]=[I_2,\mathfrak{g}_4]=I_2.$$
	
Below are the graphs $\Gamma$ and B$\Gamma$ corresponding to that decomposition:
	\begin{center}
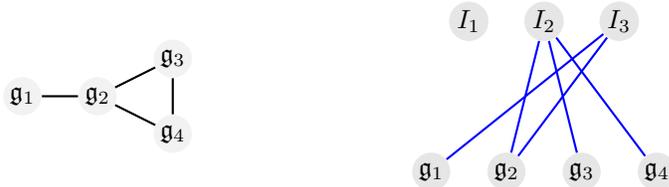
		\begin{minipage}{0.3\textwidth}
		\begin{tikzpicture}
		\tikzstyle{vertex}=[circle,fill=black!5,minimum
		size=12pt,inner sep=1pt]
		\node[vertex](11) at (5,1){$\mathfrak{g}_1$};
		\node[vertex](12) at (6,1){$\mathfrak{g}_2$};
		\node[vertex](13) at (7,1.5){$\mathfrak{g}_3$};
		\node[vertex](14) at (7,0.5){$\mathfrak{g}_4$};
		\path[draw,thick,-] (11) -- (12); 
		\path[draw,thick,-] (12) -- (13);
		\path[draw,thick,-] (14) -- (13);
		\path[draw,thick,-] (12) -- (14);
		\end{tikzpicture}	 
	\end{minipage}
	\hspace*{1cm}
	\begin{minipage}{0.3\textwidth}
		\begin{tikzpicture}
		\tikzstyle{vertex}=[circle,fill=black!10,minimum		size=12pt,inner sep=1pt]
		\node[vertex](01) at (5.5,3){$I_1$};
		\node[vertex](02) at (6.5,3){$I_2$}; 
		\node[vertex](03) at (7.5,3){$I_3$};  
		\node[vertex](11) at (5,1){$\mathfrak{g}_1$};
		\node[vertex](12) at (6,1){$\mathfrak{g}_2$};
		\node[vertex](13) at (7,1){$\mathfrak{g}_3$};
		\node[vertex](14) at (8,1){$\mathfrak{g}_4$};
		\path[draw,blue,thick,-] (02) -- (12);
		\path[draw,blue,thick,-] (02) -- (13);
		\path[draw,blue,thick,-] (02) -- (14);
		\path[draw,blue,thick,-] (03) -- (11);
		\path[draw,blue,thick,-] (03) -- (12);
		\end{tikzpicture}
	\end{minipage}
\end{center}
Although $\Gamma$ is connected, but since $[I_1,\oplus_{i=1}^4\mathfrak{g}_i]=\{0\}$, the Leibniz algebra  cannot be indecomposable. In fact, the graph B$\Gamma$ shows that such Leibniz algebra is not only indecomposable, but does not exist since $I_1$ is not being generated by any elements of the Leibniz algebra.
\end{ex}

The next statement shows the construction of an indecomposable semisimple Leibniz algebra from any finite connected bipartite graph. 

\begin{thm}\label{construction}
	Given a finite and connected bipartite graph, there exists a finite-dimensional indecomposable semisimple Leibniz algebra whose corresponding bipartite graph is the given one. 
\end{thm}
\begin{proof}
Let B$\Gamma$ be a finite, connected bipartite graph with bipartition $(V, W)$. Let us denote formally the vertices of $V$ by $v_1,\dots, v_m$ and the vertices of $W$ by $w_1,\dots, w_n$.  The essential $m \times n$ submatrix 
$$A=(a_{ij} \mid a_{ij}=1 \text{ if } v_iw_j \text{ is an edge and 0 otherwise} )_{1 \leq i \leq m, 1\leq j \leq n}$$ of the adjacency matrix of the graph $\Gamma$ contains $1$ in every row and in every column since the graph is connected. This submatrix indicates how to build an indecomposable semisimple Leibniz algebra with a corresponding bipartite graph B$\Gamma$.  Indeed, pick any simple Lie algebra $\mathfrak{g}_1,\dots, \mathfrak{g}_m$ and simple $\mathfrak{g}_i$-modules $M_{ji}$ for all $1\leq j \leq n, \ 1\leq i \leq m$. Next, define a tensor $(\oplus_{i=1}^m \mathfrak{g}_i)$-module $I_k=\otimes_{i=1}^m J_{ki}$, where $J_{ki}=M_{ki}$ if $a_{ki}=1$ and $J_{ki}=\mathbb{C}$ otherwise. By Theorem \ref{theorem_Wan} the $(\oplus_{i=1}^m \mathfrak{g}_i)$-module $I_k$ is simple for all $1\leq k \leq n$. Then $L=(\oplus_{i=1}^m \mathfrak{g}_i) \ltimes (\oplus_{k=1}^n I_k)$ is the Leibniz algebra with corresponding bipartite graph B$\Gamma$. Note that by Theorem \ref{indecomposable_iff_connected} it follows that $L$ is indecomposable.
\end{proof}

\begin{ex} Up to an isomorphism there are exactly two connected bipartite graphs with essential submatrix $A$ being a $2\times 2$ matrix:

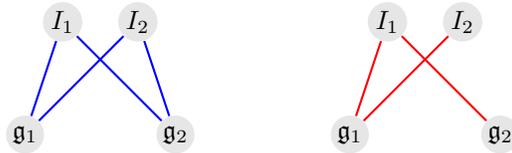
\begin{center}
	\begin{minipage}{0.30\textwidth}
		\begin{tikzpicture}
		\tikzstyle{vertex}=[circle,fill=black!10,minimum
		size=12pt,inner sep=1pt]
		\node[vertex](01) at (5,2.5){$I_1$};
		\node[vertex](02) at (6,2.5){$I_2$};  
		\node[vertex](11) at (4.5,1){$\mathfrak{g}_1$};
		\node[vertex](12) at (6.5,1){$\mathfrak{g}_2$};
		\path[draw, blue,thick,-] (01) -- (11); 
		\path[draw,blue,thick,-] (01) -- (12);
		\path[draw,blue,thick,-] (02) -- (11);
		\path[draw,blue,thick,-] (02) -- (12);
		\end{tikzpicture}
	\end{minipage}
	\begin{minipage}{0.30\textwidth}
		\begin{tikzpicture}
		\tikzstyle{vertex}=[circle,fill=black!10,minimum
		size=12pt,inner sep=1pt]
		\node[vertex](01) at (5,2.5){$I_1$};
		\node[vertex](02) at (6,2.5){$I_2$};  
		\node[vertex](11) at (4.5,1){$\mathfrak{g}_1$};
		\node[vertex](12) at (6.5,1){$\mathfrak{g}_2$};
		\path[draw,red, thick,-] (01) -- (11); 
		\path[draw, red,thick,-] (01) -- (12);
		\path[draw, red, thick,-] (02) -- (11);
		\end{tikzpicture}
	\end{minipage}\hspace*{0.2cm} 
\end{center}
Then by the construction given in the Theorem \ref{construction} the corresponding indecomposable semisimple Leibniz algebras are the following:
\begin{center}
$L_1=(\mathfrak{g}_1\oplus \mathfrak{g}_2)\ltimes \left( (J_{11}\otimes J_{12})\oplus (J_{21}\otimes J_{22}) \right),$

$L_2=(\mathfrak{g}_1\oplus \mathfrak{g}_2)\ltimes \left( (J_{11}\otimes J_{12})\oplus J_{21} \right),$
\end{center}
where $J_{pq}$ is a simple $\mathfrak{g}_q$-module (and for $L_2$, $\mathfrak{g}$-module $J_{21}\cong   J_{21}\otimes \mathbb{C}$). Note that for both $L_1$ and $L_2$ the graph $\Gamma$ is the same simple connected graph on two vertices. 
\end{ex}

\end{document}